\documentclass[11pt]{amsart} 
\usepackage[utf8]{inputenc}
\makeatletter
\@namedef{subjclassname@2010}{%
  \textup{2010} Mathematics Subject Classification}
\makeatother
\usepackage{amssymb,color,verbatim}
\newtheorem{theorem}{Theorem}
\newtheorem{cor}[theorem]{Corollary}

\newtheorem{lemma}[theorem]{Lemma}

\newtheorem{remark}[theorem]{Remark}

\newcommand{\bs}{\bigskip}
\newcommand{\n}{\noindent}

\newcommand{\Ba}{\ensuremath{\mathcal{B}}}
\newcommand{\Pa}{\ensuremath{\mathcal{P}}}

\newcommand{\F}{\ensuremath{\mathbb{F}}}
\newcommand{\ta}{\Theta}
\newcommand{\de}{\delta}
\newcommand{\al}{\alpha}
\newcommand{\be}{\beta}

\newcommand{\pg}{{\rm PG}}
\newcommand{\md}{\;{\rm mod}\;}

\begin{document}
\title{The maximum  size of a partial spread II: Upper bounds}
\author[Esmeralda N\u{a}stase and Papa Sissokho]{\tiny{Esmeralda N\u{a}stase} \\ \\
 Mathematics Department\\ Xavier University\\
Cincinnati, Ohio 45207, USA\\ \\
Papa Sissokho  \\ \\
Mathematics Department \\ Illinois State University\\ 
Normal, Illinois 61790, USA}
\thanks{nastasee@xavier.edu,  psissok@ilstu.edu}
\begin{abstract}
Let $n$ and $t$ be positive integers with $t<n$, and let $q$ be a prime power. 
A {\em partial $(t-1)$-spread} of ${\rm PG}(n-1,q)$ is a set of $(t-1)$-dimensional subspaces of ${\rm PG}(n-1,q)$ 
that are pairwise disjoint.
Let $r\equiv n\pmod{t}$ with $0\leq r<t$, and let $\ta_i=(q^i-1)/(q-1)$. We essentially prove that if $2\leq r<t\leq \ta_r$,  then the maximum size of a partial 
$(t-1)$-spread of ${\rm PG}(n-1,q)$ is bounded from  above by 
$(\ta_n-\ta_{t+r})/\ta_t+q^r-(q-1)(t-3)+1$. 
We actually  
give tighter bounds when certain divisibility conditions are satisfied. These bounds improve on
the previously known upper bound for the maximum size partial ($t-1$)-spreads of ${\rm PG}(n-1,q)$; for instance, when $\lceil\frac{\ta_r}{2}\rceil+4\leq t\leq \ta_r$ and $q>2$. 
The exact value of the maximum size partial $(t-1)$-spread has been recently determined for $t>\ta_r$ by the authors of this paper (see N\u{a}stase-Sissokho~\cite{NS}). 
\end{abstract}
\maketitle
\n  

\bs\n
\section{Introduction}\label{sec:1}
Let $n$ and $t$ be positive integers with $t<n$, and let $q$ be a prime power.
Let $\pg(n-1,q)$ denote the $(n-1)$-dimensional 
projective space over the finite field $\F_q$.  
A {\em partial $(t-1)$-spread} $S$ of $\pg(n-1,q)$ is a collection 
of $(t-1)$-dimensional subspaces of 
$\pg(n-1,q)$ that are pairwise disjoint.
If $S$ contains all the points of $\pg(n-1,q)$, then it is called 
a {\em $(t-1)$-spread}. It is well-known that a $(t-1)$-spread of $\pg(n-1,q)$ exists if and only if $t$ divides $n$ (e.g., see~\cite[p. 29]{De}).
Besides their traditional relevance to Galois geometry~\cite{EiStSz,GaSz,He1,JuSt}, partial $(t-1)$-spreads are used to build byte-correcting codes (e.g., see~\cite{Et1,HP}), $1$-perfect mixed
error-correcting codes (e.g., see~\cite{HeSc,HP}), orthogonal arrays (e.g., see~\cite{DF}), and subspace codes (e.g., see~\cite{EV,GR,KK}).

\bs\n {\bf Convention:} For the rest of the paper, we assume that $q$ is 
a prime power, and  $n$, $t$, and $r$ are integers such that $n>t>r\geq0$ and $r\equiv n\pmod{t}$.
We also use $\mu_q(n,t)$ to denote the maximum size of any partial $(t-1)$-spread of $\pg(n-1,q)$.

The problem of determining $\mu_q(n,t)$ is a long standing open problem. 
Currently, the best general upper bound for $\mu_q(n,t)$ is given by the following theorem of Drake and Freeman~\cite{DF}.

\begin{theorem}\label{DF}
If $r>0$, then $\mu_q(n,t)\leq\frac{q^n-q^{t+r}}{q^t-1}+q^r-\lfloor\omega\rfloor-1$,\\
where $2\omega=\sqrt{4q^t(q^t-q^r)+1}-(2q^t-2q^r+1)$.
\end{theorem}
The following result is attributed to Andr\'e~\cite{An} and Segre~\cite{Se} for $r=0$. For $r=1$, it is due to 
Hong and Patel~\cite{HP} when $q=2$, 
and Beutelspacher~\cite{Be} when $q>2$.
\begin{theorem}\label{BHP}
If $0\leq r<t$, then $\mu_q(n,t)\geq\frac{q^n-q^{t+r}}{q^t-1}+1$, and equality holds if $r\in\{0,1\}$.
\end{theorem}
In light of Theorem~\ref{BHP}, it was later conjectured (e.g., see~\cite{EiSt,HP}) that the value of $\mu_q(n,t)$ is given by the lower bound in Theorem~\ref{BHP}. However, this conjecture was disproved by El-Zanati, Jordon, Seelinger, Sissokho, and Spence~\cite{EJSSS} who
proved the following result.
\begin{theorem}\label{EJSSS} If $n\geq8$ and $n\md{3}=2$, then
$\mu_2(n,3)=\frac{2^n-2^{5}}{7}+2$.
\end{theorem}
Recently, Kurz~\cite{K} proved the following theorem which upholds the lower bound for $\mu_q(n,t)$ when $q=2$, $r=2$, and $t>3$. 
\begin{theorem}\label{K} If $n>t>3$ and $n\md{t}=2$, then
$\mu_2(n,t)=\frac{2^n-2^{t+2}}{2^t-1}+1$.
\end{theorem}
For any integer $i\geq1$, let  
\begin{equation}
\ta_i=(q^i-1)/(q-1).
\end{equation}

Still recently, the authors of this paper affirmed the conjecture 
(e.g., see \cite{EiSt,HP}) on the value of $\mu_q(n,t)$ for $t>\ta_r$ and any prime power $q$, by proving the following general result~(see \cite{NS}).
\begin{theorem}\label{NS}
If $t>\ta_r$, then 
$\mu_q(n,t)=\frac{q^n-q^{t+r}}{q^t-1}+1$.
\end{theorem}
In light of Theorem~\ref{NS}, it remains to determine 
the value of $\mu_q(n,t)$ for $2\leq r<t\leq \ta_r$. 
In this paper, we apply the {\em hyperplane averaging method} that we 
devised in~\cite{NS} to prove the following results\footnote{
Also see~\cite{K2} for a recent preprint in this area.}. The rest of the paper is devoted to 
their proofs.
\begin{theorem}\label{thm:mq} 
 Let $c_1\equiv(t-2)\pmod{q}$, $0\leq c_1<q$, and $c_2=\begin{cases} 
q & \mbox{ if }  q^2\mid \left((q-1)(t-2)+c_1\right)\\ 
0& \mbox{ if }  q^2\nmid \left((q-1)(t-2)+c_1\right).
\end{cases}$\\
If $2\leq r<t\leq \ta_r$, then
\[\mu_q(n,t)\leq \frac{q^n-q^{t+r}}{q^t-1}+q^r-(q-1)(t-2)-c_1+c_2.
\]
Consequently,
\[\mu_q(n,t)\leq \frac{q^n-q^{t+r}}{q^t-1}+q^r-(q-1)(t-3)+1.
\]
\end{theorem}
\begin{remark}
The best possible bound in Theorem~\ref{thm:mq} is obtained when 
$t\equiv aq+1\pmod{q^2}$, $1\leq a\leq q-1$ (equivalently, when $t\equiv 1\pmod{q}$ but $t\not\equiv 1\pmod{q^2}$).
In this case, we can check that $c_1=q-1$ and $c_2=0$, which implies that 
\[\mu_q(n,t)\leq \frac{q^n-q^{t+r}}{q^t-1}+q^r-(q-1)(t-1).
\]
This was already noted in~\cite[Lemma~$10$ and Remark~$11$]{NS} for 
$r\geq2$ and $t=\ta_r=(q^r-1)/(q-1)$. 
\end{remark}
\begin{cor}\label{cor:M} \ 
Let $f_q(n,t)$ denote the upper bound for $\mu_q(n,t)$
in Theorem~\ref{DF} and let $g_q(n,t)$denote the upper bound for $\mu_q(n,t)$
in Theorem~\ref{thm:mq}. Let $c_1$ and $c_2$ be as defined in 
Theorem~\ref{thm:mq}.
If $r\geq2$ and $2r \leq t\leq \ta_r$ then
\[g_q(n,t)-f_q(n,p)= \left\lfloor\frac{q^r}{2}\right\rfloor-(q-1)(t-2)-c_1+c_2.\]
Consequently, for $\lceil\frac{\ta_r}{2}\rceil+4\leq t\leq \ta_r$ with $q>2$, and 
for $\lceil\frac{\ta_r}{2}\rceil+5\leq t\leq \ta_r$ with $q=2$, we have
\[ 
g_q(n,t)-f_q(n,p)<0,\]
and thus the upper bound for $\mu_q(n,t)$ given in Theorem~\ref{thm:mq} is tighter than the Drake--Freeman bound in Theorem~\ref{DF}.
\end{cor}

\bs In Section~\ref{sec:2}, we present some auxiliary results from the area of subspace partitions, and in Section~\ref{sec:3} we prove Theorem~\ref{thm:mq} and Corollary~\ref{cor:M}.
\section{Subspace partitions}\label{sec:2}
Let $V=V(n,q)$ denote the vector space of dimension $n$ over $\F_q$.
For any subspace $U$ of $V$, let $U^*$ denote the set of nonzero vectors in $U$.  A  {\em $d$-subspace} of $V(n,q)$
is a $d$-dimensional subspace of $V(n,q)$; this is equivalent to a {\em $(d-1)$-subspace}  in $\pg(n-1,q)$.

A {\em subspace partition} $\Pa$ of $V$, also known as a {\em vector space partition}, 
is a collection of 
nontrivial subspaces of $V$ such that each vector of $V^*$ is in exactly one subspace of $\Pa$ (e.g., see Heden~\cite{He1} for 
a survey on subspace partitions). 
The {\em size} of a subspace partition $\Pa$, denoted by $|\Pa|$, is the number of subspaces in $\Pa$. 

Suppose that there are $s$ distinct integers, $d_s>\dots>d_1$, that occur as dimensions of subspaces
in a subspace partition $\Pa$, 
and let $n_i$ denote the number of $i$-subspaces in $\Pa$. Then the expression $[d_s^{n_{d_s}},\ldots,d_1^{n_{d_1}}]$ is called the {\em type} of $\Pa$.  
\begin{remark}\label{rmk:eq}
A partial $(t-1)$-spread of $\pg(n-1,q)$ of size $n_t$ is a partial $t$-spread of $V(n,q)$ of size $n_t$. This is 
equivalent to a subspace partition of $V(n,q)$ of type 
$[t^{n_t},1^{n_1}]$, where $n_1=\ta_n-n_t\ta_{t}$.
We will use this subspace partition formulation in the proof 
of Lemma~\ref{lem:main}.
\end{remark}

\bs Also, we will use the following theorem due to Heden~\cite{He-T}  in the proof of Lemma~\ref{lem:main}.
\begin{theorem}\label{He-T}\cite[Theorem~$1$]{He-T}
Let ${\mathcal P}$ be a subspace partition of $V(n,q)$ of type $[d_s^{n_{d_s}},\ldots,d_1^{n_{d_1}}]$, where $d_s>\ldots>d_1$. Then, 
\begin{enumerate}
\item[(i)] if $q^{d_2-d_1}$ does not divide $n_{d_1}$ and if $d_2<2d_1$, 
then $n_{d_1}\geq q^{d_1}+1$.
\item[(ii)] if $q^{d_2-d_1}$ does not divide $n_{d_1}$ and $d_2\geq 2d_1$, then either
 $n_{d_1}=(q^{d_2}-1)/(q^{d_1}-1)$ or $n_{d_1}>2q^{d_2-d_1}$.
\item[(iii)] if $q^{d_2-d_1}$ divides $n_{d_1}$ and $d_2<2d_1$, then 
 $n_{d_1}\geq q^{d_2}-q^{d_1}+q^{d_2-d_1}$.
\item[(iv)] if $q^{d_2-d_1}$ divides $n_{d_1}$ and $d_2\geq 2d_1$, then $n_{d_1}\geq q^{d_2}$.
\end{enumerate}
\end{theorem}

\bs To state the next lemmas, we need the following definitions. Recall that for any integer 
$i\geq1$,
\[\ta_i=(q^i-1)/(q-1).\]

Then, for $i\geq 1$, $\ta_i$ is the number of $1$-subspaces in an $i$-subspace of $V(n,q)$. 
Let $\Pa$ be a subspace partition of $V=V(n,q)$ of type 
$[d_s^{n_{d_s}},\ldots,d_1^{n_{d_1}}]$.
For any hyperplane $H$ of $V$, let $b_{H,d}$ be the number of $d$-subspaces in $\Pa$ that are contained in $H$ and set 
$b_{H}=[b_{H,d_s},\ldots,b_{H,d_1}]$.  
Define the set $\Ba$ of {\em hyperplane types} as follows:
\[\Ba=\{b_H:\; \mbox{$H$ is a hyperplane of $V$}\}.\]
For any $b \in \Ba$, let $s_{b}$ denote the number of hyperplanes of $V$ of type $b$.

\bs
We will also use Lemma~\ref{lem:HeLe0} and Lemma~\ref{lem:HeLe1} 
by Heden and Lehmann~\cite{HeLe} in the proof 
of Lemma~\ref{lem:main}.
\begin{lemma}\label{lem:HeLe0}\cite[Equation~$(1)$]{HeLe}
Let $\Pa$ be a subspace partition of $V(n,q)$ of type 
$[d_s^{n_{d_s}},\ldots,d_1^{n_{d_1}}]$. If $H$ is a hyperplane 
of $V(n,q)$ and $b_{H,d}$ is as defined above, then
\[|\Pa|=1+\sum\limits_{i=1}^s b_{H,d_i}q^{d_i}.\]
\end{lemma}
\begin{lemma}\label{lem:HeLe1}\cite[Equation~$(2)$ and Corollary $5$]{HeLe}
Let $\Pa$ be a subspace partition of $V(n,q)$, and let $\Ba$ and $s_b$ be as defined above. Then 
\[\sum\limits_{b\in \Ba} s_b=\ta_n,\]
and for $1\leq d\leq n-1$, we have
\[\sum\limits_{b\in \Ba}b_ds_b=n_d \ta_{n-d}.\]
\end{lemma}
\section{Proofs of the main results}\label{sec:3}
Recall that $q$ is a prime power, and  $n,t,$ and $r$ are integers 
such that $n>t>r\geq0$, and $r\equiv n\pmod{t}$. To prove our main result, 
we first need to prove the following two technical lemmas.
\begin{lemma}\label{lem:de} Let $x$ be an integer such that $0<x<q^r$.
For any positive integer $i$, let $\de_i=q^i\cdot\lceil xq^{-i}\ta_{i}\rceil-x\ta_i$. Then the following properties hold:
\begin{enumerate}
\item[(i)] $\lceil xq^{-t}\ta_{t}\rceil=\lceil\frac{x}{q-1}\rceil$. 
\item[(ii)] for $1\leq i\leq t$, we have $0\leq \de_{i}<q^i$, $q\mid(x+\de_{i+1})$, and $\de_{i}=q^{-1}(x+\de_{i+1})\md{q^{i}}$.
\item[(iii)]  $\de_i=0$ if and only if $q^i\mid x$.
\end{enumerate}
\end{lemma}
\begin{proof}\
Let $\al$ and $\be$ be integers such that $x=\al(q-1)+\be$, $\al\geq 0$, and $0\leq\be<q-1$. Since $0<x<q^r$ and $r<t$ hold by hypothesis, it follows that
\begin{equation}\label{eq:de0}
0\leq \al <x <q^r<q^t \mbox{ and }\al(q-1)\leq x <q^r<q^t.
\end{equation}

If $\be=0$, then by ~\eqref{eq:de0}, we obtain
\begin{align}\label{eq:de1.1a}
\left\lceil xq^{-t}\ta_{t}\right\rceil=\left\lceil\frac{\al(q^t-1)}{q^t}\right\rceil
= \left\lceil\al-\frac{\al}{q^t}\right\rceil=\al=\left\lceil\frac{x}{q-1}\right\rceil.
\end{align}
Now suppose $1\leq \be <q-1$. First, since $\be \geq 1$, it follows from~\eqref{eq:de0} that
\begin{align}\label{eq:de1.2a}
\left\lceil xq^{-t}\ta_{t}\right\rceil=\left\lceil\frac{[\al(q-1)+\be](q^t-1)}{q^t(q-1)}\right\rceil 
&\geq \left\lceil\frac{[\al(q-1)+1](q^t-1)}{q^t(q-1)}\right\rceil\cr
&=\left\lceil \al+ \frac{(q^t-1)-\al(q-1)}{q^t(q-1)}\right\rceil\cr
&= \al +1.
\end{align}
Second, since $\be< q-1$, it follows 
from~\eqref{eq:de0} and the properties of the ceiling function that
\begin{align}\label{eq:de1.3a}
\left\lceil xq^{-t}\ta_{t}\right\rceil=\left\lceil\frac{[\al(q-1)+\be](q^t-1)}{q^t(q-1)}\right\rceil 
\leq \left\lceil\frac{(\al+1)(q^t-1)}{q^t}\right\rceil 
=\left\lceil \al+1 -\frac{\al+1}{q^t}\right\rceil =\al+1.
\end{align}
Then \eqref{eq:de1.2a} and \eqref{eq:de1.3a} imply that for $1\leq \be<q-1$,
\[
\lceil xq^{-t}\ta_{t}\rceil = \al+1=\left\lceil\frac{x}{q-1}\right\rceil,
\]
 which completes the proof of $(i)$.

\bs We now prove $(ii)$. Since $0\leq \lceil a\rceil-a<1$ holds for any real number $a$,
we have
\[
0\leq \lceil q^{-i}x\ta_{i}\rceil-q^{-i}x\ta_{i}<1 \Longrightarrow \de_i=q^i\lceil xq^{-i}\ta_{i}\rceil -x\ta_{i}<q^i \mbox{ and $\de_i \geq 0$.}
\]
By the definition of $\de_i$, 
we have that
\[ x+\de_{i+1}=x+q^{i+1}\cdot\lceil xq^{-i-1}\ta_{i+1}\rceil-x\ta_{i+1}=q(q^{i}\cdot\lceil xq^{-i-1}\ta_{i+1}\rceil-x\ta_{i}),
\]
and thus,
\begin{align}\label{eq:de2}
q^{-1}(x+\de_{i+1})
&\equiv q^{i}\cdot\lceil xq^{-i-1}\ta_{i+1}\rceil-x\ta_{i}\cr
&\equiv -x\ta_{i}\cr
&\equiv q^{i} \cdot \lceil xq^{-i}\ta_{i}\rceil -x\ta_{i}\cr
&\equiv \de_{i}\pmod{q^{i}}.
\end{align}

\bs Finally, we prove $(iii)$.  Since $\gcd(q^i,\ta_i)=1$
for any positive integer $i$, we have
\[\de_i=q^i\cdot\lceil xq^{-i}\ta_{i}\rceil-x\ta_i=0
\Longleftrightarrow \lceil xq^{-i}\ta_{i}\rceil=xq^{-i}\ta_i
\Longleftrightarrow q^i| x.
\]
\end{proof}
\bs We now prove our main lemma. 
\begin{lemma}\label{lem:main} 
Let $x$ be a positive integer such that $q\mid x$ and $q^2\nmid x$. 
Let $\ell=(q^{n-t}-q^{r})/(q^t-1)$.
 If $r\geq 2$ and $t\geq\ta_r-\lceil x/(q-1)\rceil+2$, then $\mu_q(n,t)\leq \ell  q^t+x$.
\end{lemma}
\begin{proof} 
 If $x\geq q^r$, then Theorem~\ref{DF} implies the 
nonexistence of a partial $t$-spread of size $\ell q^t+x$. 
Thus, we can assume that $x<q^r$.

Recall that $\ta_i=(q^i-1)/(q-1)$ for any integer $i\geq 1$. 
For an integer $i$, with $2\leq i\leq t$, let 
\begin{equation}\label{eq:dli} 
\de_i=q^i\cdot\lceil xq^{-i}\ta_{i}\rceil-x\ta_i.
\end{equation}
Applying Lemma~\ref{lem:de}(i), we let
\begin{equation}\label{eq:dlt1}
h:=\lceil q^{-t}x\ta_{t}\rceil=\left\lceil\frac{x}{q-1}\right\rceil.
\end{equation}

The proof is by contradiction. So assume that  
$\mu_q(n,t)>\ell  q^t+x$.  Then $\pg(n-1,q)$ has a
 $(t-1)$-partial spread of size $\ell  q^t+1+x$.
Thus, it follows from Remark~\ref{rmk:eq} that there exists a subspace partition $\Pa_0$ of $V(n,q)$ of type $[t^{n_t},1^{n_1}]$, with
\begin{align}\label{eq:ntn1}
&n_t=\ell  q^t+1+x\mbox{, and } \cr
&n_1=q^t\ta_r-x\ta_{t}=q^t(\ta_r-\lceil q^{-t}x\ta_{t}\rceil)+(q^t\lceil q^{-t}x\ta_{t}\rceil-x\ta_{t})=q^t(\ta_r-h)+\de_t,
\end{align}
where $h$ is given by~\eqref{eq:dlt1} and $\de_t$ is given by~\eqref{eq:dli}.

We will prove by induction that for each integer $j$ with $0\leq j\leq t-2$,  there exists  a subspace partition $\Pa_{j}$ of $H_j\cong V(n-j,q)$ of type 
\begin{equation}\label{ih}
[t^{m_{j,t}},(t-1)^{m_{j,t-1}}, \dots,(t-j)^{m_{j,t-j}},1^{m_{j,1}}],
\end{equation}
where $m_{j,t},\ldots,m_{j,t-j}$ are nonnegative integers such that 
\begin{equation}\label{prp1}
\sum_{i=t-j}^t m_{j,i}=n_t=\ell  q^t+1+x,
\end{equation} 
and where $m_{j,1}$ and $c_j$ are integers such that
\begin{equation}\label{prp2}
m_{j,1}=c_jq^{t-j}+\de_{t-j},\mbox{ and } 0\leq c_j\leq\max\{\ta_r-h-j,0\}.
\end{equation} 
The base case, $j=0$, holds since $\Pa_0$ is 
a subspace partition of $H_0=V(n,q)$ with type $[t^{n_t},1^{n_1}]$, 
and letting $m_{0,t}=n_t$ and $m_{0,1}=n_1$,  $\Pa_0$ is of type
given in~\eqref{ih}, and 
 it satisfies the properties given in~\eqref{prp1} 
and~\eqref{prp2}. 

For the inductive step, suppose that for some $j$, with $0\leq j<t-2$,
 we have constructed a subspace partition $\Pa_{j}$ of $H_{j}\cong V(n-j,q)$ of the type given in~\eqref{ih}, and with the properties given in~\eqref{prp1} and~\eqref{prp2}.
We then use Lemma~\ref{lem:HeLe1} to determine the average, $b_{avg,1}$, of the values $b_{H,1}$ over all hyperplanes $H$ of $H_{j}$. We have
\begin{align}\label{eq:5}
b_{avg,1}:=\frac{m_{j,1} \ta_{n-1-j}}{ \ta_{n-j}}
&=\left(c_jq^{t-j}+\de_{t-j}\right)\left(\frac{q^{n-1-j}-1}{q^{n-j}-1}\right)\cr
&<(c_jq^{t-j}+\de_{t-j})q^{-1}\cr
&=c_jq^{t-j-1}+q^{-1}\de_{t-j}.
\end{align} 
It follows from~\eqref{eq:5} that there exists a hyperplane $H_{j+1}$ of $H_{j}$ with 
\begin{equation}\label{eq:6}
b_{H_{j+1},1}\leq b_{avg,1}<c_jq^{t-j-1}+q^{-1}\de_{t-j}.
\end{equation}
Next, we apply Lemma~\ref{lem:HeLe0} to the subspace partition 
$\Pa_j$ and the hyperplane $H_{j+1}$ of $H_{j}$ to obtain:
\begin{align}\label{eq:7.1}
1+b_{H_{j+1},1}\;q+\sum_{i=t-j}^t b_{H_{j+1},i}\;q^i
&=|\Pa_j|\cr
&=n_t+m_{j,1}\cr
&=\ell  q^{t}+ 1+x+c_jq^{t-j}+\de_{t-j},
\end{align}
where $0\leq c_j\leq \max\{\ta_r-h-j,0\}$. Simplifying~\eqref{eq:7.1} yields
\begin{align}\label{eq:7.2}
b_{H_{j+1},1}+\sum_{i={t-j}}^{t} b_{H_{j+1},i}\;q^{i-1}
&=\ell  q^{t-1}+c_jq^{t-j-1}+q^{-1}(x+\de_{t-j}).
\end{align}
Then, it follows from~Lemma~\ref{lem:de}(ii) and~\eqref{eq:7.2} that
\begin{equation}\label{eq:8}
b_{H_{j+1},1}\equiv q^{-1}(x+\de_{t-j})\equiv \de_{t-j-1} \pmod{q^{t-j-1}}.
\end{equation}
Since $0\leq q^{-1}\de_{t-j}<q^{t-j-1}$ by~Lemma~\ref{lem:de}(ii), it follows from~\eqref{eq:6} and~\eqref{eq:8} that there exists a nonnegative integer $c_{j+1}$ such that 
\begin{align}\label{prp2+}
&b_{H_{j+1},1}=c_{j+1}q^{t-j-1}+\de_{t-j-1} 
\mbox{ and } \cr
&0\leq c_{j+1}\leq\max\{c_j-1,0\}\leq  \max\{\ta_r-h-j-1,0\}.
\end{align}
Let $\Pa_{j+1}$ be the subspace partition of $H_{j+1}$ defined by:
\[\Pa_{j+1}=\{W\cap H_{j+1}:\; W\in \Pa_{j}\},
\]
and by the definition made in~\eqref{prp2+}, let $m_{j+1,1}=b_{H_{j+1},1}$.
Since $t-j>2$ and $\dim(W\cap H_{j+1})\in\{\dim W,\dim W-1\}$ for 
each $W\in\Pa_{j}$, it follows that $\Pa_{j+1}$ is a subspace partition of $H_{j+1}$ of type 
\begin{equation}\label{ih+}
[t^{m_{j+1,t}},(t-1)^{m_{j+1,t-1}}, \dots,(t-j-1)^{m_{j+1,t-j-1}},1^{m_{j+1,1}}],
\end{equation}
where $m_{j+1,t},m_{j+1,t-1},\ldots,m_{j+1,t-j-1}$ 
are nonnegative integers such that
\begin{equation}\label{prp1+}
\sum_{i=t-j-1}^t m_{j+1,i}=\sum_{i=t-j}^t m_{j,i}=n_t.
\end{equation}
  
The inductive step follows since $\Pa_{j+1}$ is a subspace partition of $H_{j+1}\cong V(n-j-1,q)$
of the type given in~\eqref{ih+}, which satisfies the conditions in~\eqref{prp2+} and~\eqref{prp1+}. 

\bs Thus far, we have shown that the desired subspace partition 
$\Pa_j$ of $H_j$ exists for any integer
$j$ such that $0\leq j\leq t-2$. 
Since $q^2\nmid x$ by hypothesis, Lemma~\ref{lem:de}(iii) implies that 
$\de_{t-j}\not=0$ for $j\in[0,t-2]$. Thus, $m_{j,1}=c_jq^{t-j}+\de_{t-j}\not=0$ for  $j\in[0,t-2]$. If $j\in[\ta_r-h,t-2]$, then it follows 
from~\eqref{prp2} that $c_j=0$, and thus, $m_{j,1}=\de_j\not=0$.
In particular, since $t\geq\ta_r-h+2$, we have $c_{t-2}=0$
and $m_{t-2,1}=\de_2\not=0$.
For the final part of the proof, we set $j=t-2$, and then show that the existence of the subspace partition $\Pa_{t-2}$ of $H_{t-2}$ leads to a contradiction.

It follows from the above observations and Lemma~\ref{lem:de}(ii) that
\begin{equation}\label{mk1}
m_{t-2,1}=\de_{2}=q^2\lceil xq^{-2}\ta_{2}\rceil-x\ta_{2}\mbox{ and } 0<\de_2<q^2.
\end{equation}
Since $m_{t-1,2}>0$, the smallest dimension of a subspace in $\Pa_{t-2}$ 
is $1$. So let $s\geq 2$ be the second smallest dimension of a subspace in 
$\Pa_{t-2}$. (Note that the existence of $s$ follows from~\eqref{prp1}.) 
To derive the final contradiction, we consider the following cases.

\bs\n {\bf Case 1:} $s\geq3$.

Then by applying Theorem~\ref{He-T}(ii)\&(iv)
to the subspace partition $\Pa_{t-2}$ with $d_2=s$  and $d_1=1$, we obtain  $m_{t-2,1}\geq \min\{(q^s-1)/(q-1),2q^{s-1},q^s\}>q^2$, which contradicts the fact that $m_{t-2,1}<q^2$ given by~\eqref{mk1}. 

\bs\n {\bf Case 2:} $s=2$.

Since $q\mid x$ by hypothesis, it follows from~\eqref{mk1} that $q\mid m_{t-2,1}$. Thus, by applying Theorem~\ref{He-T}(iv) to $\Pa_{t-2}$ with $d_2=s=2$  and $d_1=1$, we obtain  
$m_{t-2,1}\geq q^2$, which contradicts the fact that $m_{t-2,1}<q^2$ given 
by~\eqref{mk1}.
\end{proof}
We are now ready to prove Theorem~\ref{thm:mq} and Corollary~\ref{cor:M}.
\begin{proof}[Proof of Theorem~\ref{thm:mq}]
Recall that
\begin{equation}\label{eq:c2-2}
\mbox{$c_1\equiv t-2\pmod{q}$, $0\leq c_1<q$, and } c_2=\begin{cases} 
q & \mbox{if }  q^2\mid \left((q-1)(t-2)+c_1\right),\\ 
0& \mbox{if }  q^2\nmid \left((q-1)(t-2)+c_1\right).
\end{cases}
\end{equation}
Define
\begin{equation}\label{eq:x1}
x:=q^r-(q-1)(t-2)-c_1+c_2.
\end{equation}
Since $r\geq 2$, it follows from~\eqref{eq:c2-2} and~\eqref{eq:x1} that:
\begin{enumerate}
\item[(a)] If $q^2\mid \left((q-1)(t-2)+c_1\right)$, then $c_2=q$, 
and also, $q^2\mid \left(q^r-(q-1)(t-2)-c_1\right)$. 
Thus, $x\equiv q\not\equiv 0\pmod{q^2}$.
\item[(b)] If $q^2\nmid \left((q-1)(t-2)+c_1\right)$, then 
$c_2=0$, and also, $q^2\nmid \left(q^r-(q-1)(t-2)-c_1\right)$. Thus, 
$x=q^r-(q-1)(t-2)-c_1\not\equiv 0\pmod{q^2}$.
\end{enumerate}
\n Thus, $q^2\nmid x$ holds in all cases.

Also, since $c_1\equiv t-2\pmod{q}$ by~\eqref{eq:c2-2}, we have 
$t-2=\alpha q+c_1$ for some nonnegative integer $\alpha$. Thus, it follows from~\eqref{eq:x1} that 
\begin{equation}\label{eq:x2} 
x=q^r-\alpha q(q-1)-c_1q+c_2.
\end{equation}
Since $c_2\in\{0,q\}$ by~\eqref{eq:c2-2}, it follows from~\eqref{eq:x2} that 
$q\mid x$.

Moreover, since $0\leq c_1\leq q-1$ and $c_2 \in \{0, q\}$, 
we obtain
\begin{align}\label{eq:t}
&x=q^r-(q-1)(t-2)-c_1+c_2\geq q^r-(q-1)(t-2)-(q-1)\cr
&\Longrightarrow
\frac{x}{q-1}\geq \frac{q^r-1}{q-1}+\frac{1}{q-1}-t+1\cr
&\Longrightarrow
\left\lceil \frac{x}{q-1}\right\rceil\geq \frac{q^r-1}{q-1}-t+2\cr
&\Longrightarrow
t\geq \ta_r-\left\lceil \frac{x}{q-1}\right\rceil+2.
\end{align}
Since the hypothesis holds from the above observations, Lemma~\ref{lem:main} yields
\[\mu_q(n,t)\leq \ell q^t+x=\frac{q^n-q^{t+r}}{q^t-1}+q^r-(q-1)(t-2)-c_1+c_2.\]
Moreover, since $-q+1\leq -c_1+c_2\leq q$, it follows that
\begin{align*}
\mu_q(n,t)
&\leq \frac{q^n-q^{t+r}}{q^t-1}+q^r-(q-1)(t-2)-c_1+c_2 \cr
&\leq \frac{q^n-q^{t+r}}{q^t-1}+q^r-(q-1)(t-2)+q\cr
&=\frac{q^n-q^{t+r}}{q^t-1}+q^r-(q-1)(t-3)+1,
\end{align*}
which concludes the proof of Theorem~\ref{thm:mq}.
\end{proof}
\begin{proof}[Proof of Corollary~\ref{cor:M}]
 Let $f_q(n,t)$ and $g_q(n,t)$ be as defined in the statement of the corollary. Then 
\begin{equation}\label{eq:c1}
g_q(n,t)=\frac{q^n-q^{t+r}}{q^t-1}+q^r-(q-1)(t-2)-c_1+c_2,
\end{equation}
where $c_1$ and $c_2$ are as in~\eqref{eq:c2-2}, and 
\begin{equation}\label{eq:c2}
f_q(n,t)=\frac{q^n-q^{t+r}}{q^t-1}+q^r-\lfloor\omega\rfloor-1,
\end{equation}
where $2\omega=\sqrt{4q^t(q^t-q^r)+1}-(2q^t-2q^r+1)$.

If $r\geq1$ and $t\geq 2r$, then it is straightforward 
to show that (e.g.,see~\cite[Lemma~2]{K}) 
\begin{equation}\label{eq:c3}
\lfloor\omega\rfloor=\left\lfloor\frac{q^r-2}{2}\right\rfloor=\left\lfloor\frac{q^r}{2}\right\rfloor-1.
\end{equation}
Now it follows from~\eqref{eq:c1}--\eqref{eq:c3} that if $t\geq 2r$, then 
\begin{equation}\label{eq:c4}
g_q(n,t)-f_q(n,p)= \left\lfloor\frac{q^r}{2}\right\rfloor-(q-1)(t-2)-c_1+c_2.\end{equation}

We now prove the second part of the corollary for $q>2$.
If $\lceil\frac{\ta_r}{2}\rceil+4\leq t\leq \ta_r$, then by applying~\eqref{eq:c4} with $0\leq c_1<q$ and $c_2\in\{0,q\}$, we obtain  
\begin{align*} 
g_q(n,t)-f_q(n,p)
&\leq \left\lfloor\frac{q^r}{2}\right\rfloor-(q-1)(t-2)+q\cr 
&\leq \left\lfloor\frac{q^r}{2}\right\rfloor-(q-1)\left(\left\lceil\frac{\ta_r}{2}\right\rceil+2\right)+q\cr
&=\left\lfloor\frac{q^r}{2}\right\rfloor-(q-1)\left\lceil\frac{q^r-1}{2(q-1)}\right\rceil-q+2\cr
&\leq\frac{q^r}{2}-(q-1)\left(\frac{q^r-1}{2(q-1)}\right)-q+2\cr
&=5/2-q<0 \quad \mbox{ (since $q>2$)}. \notag
\end{align*}

If $q=2$, then by doing the same analysis as above with $t\geq \left\lceil\frac{\ta_r}{2}\right\rceil+5$ instead of $t\geq \left\lceil\frac{\ta_r}{2}\right\rceil+4$, we obtain $g_q(n,t)-f_q(n,p)<0$. This completes the proof of the corollary.

\bs\n{\bf Acknowledgement:}\
We thank the referees for their detailed comments, suggestions, and corrections which have greatly improved the paper.
\end{proof}

\end{document}